\documentclass[reqno,a4paper]{amsart}
\usepackage{amssymb}
\usepackage[colorlinks=true,hypertex]{hyperref}
\allowdisplaybreaks[4]
\theoremstyle{plain}
\newtheorem{thm}{Theorem}

\theoremstyle{remark}
\newtheorem{rem}{Remark}

\DeclareMathOperator{\arcsinh}{arcsinh}
\date{Commenced on 10 March 2009 and completed on 11 March 2009 in Jiaozuo}
\date{}

\begin{document}

\title{Monotonicity results and bounds for the inverse hyperbolic sine}

\author[F. Qi]{Feng Qi}
\address[F. Qi]{Research Institute of Mathematical Inequality Theory, Henan Polytechnic University, Jiaozuo City, Henan Province, 454010, China}
\email{\href{mailto: F. Qi <qifeng618@gmail.com>}{qifeng618@gmail.com}, \href{mailto: F. Qi <qifeng618@hotmail.com>}{qifeng618@hotmail.com}, \href{mailto: F. Qi <qifeng618@qq.com>}{qifeng618@qq.com}}
\urladdr{\url{http://qifeng618.spaces.live.com}}

\author[B.-N. Guo]{Bai-Ni Guo}
\address[B.-N. Guo]{School of Mathematics and Informatics, Henan Polytechnic University, Jiaozuo City, Henan Province, 454010, China}
\email{\href{mailto: B.-N. Guo <bai.ni.guo@gmail.com>}{bai.ni.guo@gmail.com}, \href{mailto: B.-N. Guo <bai.ni.guo@hotmail.com>}{bai.ni.guo@hotmail.com}}
\urladdr{\url{http://guobaini.spaces.live.com}}

\begin{abstract}
In this paper, we present monotonicity results of a function involving to the inverse hyperbolic sine. From these, we derive some inequalities for bounding the inverse hyperbolic sine.
\end{abstract}

\keywords{bound, inverse hyperbolic sine, monotonicity, minimum}

\subjclass[2000]{Primary 26A48, 33B10; Secondary 26D05}

\thanks{This paper was typeset using \AmS-\LaTeX}

\maketitle

\section{Introduction and main results}

In \cite[Theorem~1.9 and Theorem~1.10]{Zhu-New-Arc-Hyper}, the following inequalities were established: For $0\le x\le r$ and $r>0$,
the double inequality
\begin{equation}\label{zhu-arcsinh-ineq-1}
\frac{(a+1)x}{a+\sqrt{1+x^2}\,}\le\arcsinh x\le \frac{(b+1)x}{b+\sqrt{1+x^2}\,}
\end{equation}
holds true if and only if $a\le2$ and
\begin{equation}\label{zhu-cond-1}
b\ge\frac{\sqrt{1+r^2}\,\arcsinh r-r}{r-\arcsinh r}.
\end{equation}
\par
The aim of this paper is to elementarily generalize the inequality~\eqref{zhu-arcsinh-ineq-1} to monotonicity results and to deduce more inequalities.
\par
Our results may be stated as the following theorems.

\begin{thm}\label{Arc-Hyperbolic-Sine-thm1}
For $\theta\in\mathbb{R}$, let
\begin{equation}\label{f-theta(x)-dfn}
f_\theta(x)=\frac{\bigl(\theta+\sqrt{1+x^2}\,\bigr)\arcsinh x}{x},\quad  x>0.
\end{equation}
\begin{enumerate}
\item
When $\theta\le2$, the function $f_\theta(x)$ is strictly increasing;
\item
When $\theta>2$, the function $f_\theta(x)$ has a unique minimum.
\end{enumerate}
\end{thm}

As straightforward consequences of Theorem~\ref{Arc-Hyperbolic-Sine-thm1}, the following inequalities are inferred.

\begin{thm}\label{Arc-Hyperbolic-Sine-thm2}
Let $r>0$.
\begin{enumerate}
\item
For $\theta\le2$, the double inequality
\begin{equation}\label{Arc-Hyperbolic-Sine-ineq1}
\frac{(1+\theta)x}{\theta+\sqrt{1+x^2}\,}<\arcsinh x\le \frac{\bigl[\bigl(\theta+\sqrt{1+r^2}\,\bigr)(\arcsinh r)/r\bigr]x}{\theta+\sqrt{1+x^2}\,}
\end{equation}
holds true on $(0,r]$, where the constants $1+\theta$ and $\frac{\left(\theta+\sqrt{1+r^2}\,\right)\arcsinh r}{r}$ in~\eqref{Arc-Hyperbolic-Sine-ineq1} are the best possible.
\item
For $\theta>2$, the double inequality
\begin{equation}\label{Arc-Hyperbolic-Sine-ineq2}
\frac{4\bigl(1-1/{\theta^2}\bigr)x}{\theta+\sqrt{1+x^2}\,}\le \arcsinh x \le \frac{\max\bigl\{1+\theta, \bigl(\theta+\sqrt{1+r^2}\,\bigr)(\arcsinh r)/r\bigr\}x} {\theta+\sqrt{1+x^2}\,}.
\end{equation}
\end{enumerate}
\end{thm}

\begin{rem}
Replacing $\arcsinh x$ by $x$ in~\eqref{Arc-Hyperbolic-Sine-ineq1} and~\eqref{Arc-Hyperbolic-Sine-ineq2} yields
\begin{multline}
\left.\begin{aligned}
\frac{\bigl[\bigl(\theta+\sqrt{1+r^2}\,\bigr)(\arcsinh r)/r\bigr]\sinh x}{\theta+\cosh x}&\\
\frac{\max\bigl\{1+\theta, \bigl(\theta+\sqrt{1+r^2}\,\bigr)(\arcsinh r)/r\bigr\}\sinh x} {\theta+\cosh x}&
\end{aligned}\right\}>x\\
>\begin{cases}
\dfrac{(1+\theta)\sinh x}{\theta+\cosh x},&\theta\le2\\[0.7em]
\dfrac{4\bigl(1-1/{\theta^2}\bigr)\sinh x}{\theta+\cosh x},&\theta>2
\end{cases}
\end{multline}
for $x\in(0,\arcsinh r)$. This can be regarded as Oppenheim type inequalities for the hyperbolic sine and cosine functions. For information on Oppenheim's double inequality for the sine and cosine functions, please refer to~\cite{Oppeheim-Sin-Cos.tex} and closely related references therein.
\end{rem}

\begin{rem}
It is clear that the left-hand side inequality in~\eqref{Arc-Hyperbolic-Sine-ineq1} recovers the left-hand side inequality in~\eqref{zhu-arcsinh-ineq-1} while the right-hand side inequalities in~\eqref{zhu-arcsinh-ineq-1} and~\eqref{Arc-Hyperbolic-Sine-ineq1} do not include each other.
\end{rem}

\begin{rem}
By similar approach to prove our theorems in next section, we can procure similar monotonicity results and inequalities for the inverse hyperbolic cosine.
\end{rem}

\section{Proof of theorems}

Now we prove our theorems elementarily.

\begin{proof}[Proof of Theorem~\ref{Arc-Hyperbolic-Sine-thm1}]
Direct differentiation yields
\begin{align*}
f_\theta'(x)&=\frac1{x^2}\biggl(\theta+\frac1{\sqrt{x^2+1}\,}\biggr) \biggl[\frac{x\bigl({\theta}/{\sqrt{x^2+1}\,}+1\bigr)} {\theta+1/{\sqrt{x^2+1}\,}}-\arcsinh x\biggr]\\
&\triangleq\frac1{x^2}\biggl(\theta+\frac1{\sqrt{x^2+1}\,}\biggr)h_{\theta}(x),\\
h_{\theta}'(x)&=\frac{x^2 \bigl(2-\theta^2+\theta\sqrt{x^2+1}\,\bigr)} {\sqrt{x^2+1}\,\bigl(\theta\sqrt{x^2+1}\,+1\bigr)^2}\\
&\triangleq\frac{x^2q_x(\theta)} {\sqrt{x^2+1}\,\bigl(\theta\sqrt{x^2+1}\,+1\bigr)^2}.
\end{align*}
The function $q_x(\theta)$ has two zeros
$$
\theta_1(x)=\frac{\sqrt{1+x^2}\,-\sqrt{9+x^2}\,}2\quad \text{and}\quad \theta_2(x)=\frac{\sqrt{1+x^2}\,+\sqrt{9+x^2}\,}2.
$$
They are strictly increasing and have the bounds $-1\le\theta_1(x)<0$ and $\theta_2(x)\ge2$ on $(0,\infty)$. As a result, under the condition $\theta\not\in(-1,0)$,
\begin{enumerate}
\item
when $\theta\le-1$, the function $q_x(\theta)$ and the derivative $h_{\theta}'(x)$ are negative, and so the function $h_{\theta}(x)$ is strictly decreasing on $(0,\infty)$;
\item
when $0\le\theta\le2$, the function $q_x(\theta)$ and the derivative $h_{\theta}'(x)$ is positive, and so the function $h_{\theta}(x)$ is strictly increasing on $(0,\infty)$;
\item
when $\theta>2$, the function $q_x(\theta)$ and the derivative $h_{\theta}'(x)$ have a unique zero which is the unique minimum point of $h_{\theta}(x)$.
\end{enumerate}
Furthermore, since $\lim_{x\to\infty}h_{\theta}(x)=\infty$ for $\theta\ge0$ and $\lim_{x\to0^+}h_{\theta}(x)=0$, it follows that
\begin{enumerate}
\item
when $\theta\le-1$, the function $h_{\theta}(x)$ is negative, and so the derivative $f_\theta'(x)$ is positive, that is, the function $f_\theta(x)$ is strictly increasing on $(0,\infty)$;
\item
when $0\le\theta\le2$, the function $h_{\theta}(x)$ is positive, and so the derivative $f_\theta'(x)$ is also positive, accordingly, the function $f_\theta(x)$ is strictly increasing on $(0,\infty)$;
\item
when $\theta>2$, the function $h_{\theta}(x)$ and the derivative $f_\theta'(x)$ have a unique zero which is the unique minimum point of the function $f_\theta(x)$ on $(0,\infty)$.
\end{enumerate}
\par
On the other hand, when $\theta\in(-1,0)$, we have
\begin{align*}
\bigl[x^2f_\theta'(x)\bigr]'&=\frac{x^2\bigl[\sqrt{x^2+1}\, +(\arcsinh x)/x-\theta\bigr]} {(x^2+1)^{3/2}}>0
\end{align*}
which means that the function $x^2f_\theta'(x)$ is strictly increasing on $(0,\infty)$. From the limit $\lim_{x\to0^+}\bigl[x^2f_\theta'(x)\bigr]=0$, it is derived that the function $x^2f_\theta'(x)$ is positive. Hence the function $f_\theta(x)$ is strictly increasing on $(0,\infty)$.
The proof of Theorem~\ref{Arc-Hyperbolic-Sine-thm1} is complete.
\end{proof}

\begin{proof}[Proof of Theorem~\ref{Arc-Hyperbolic-Sine-thm2}]
Since $\lim_{x\to0^+}f_\theta(x)=1+\theta$, by Theorem~\ref{Arc-Hyperbolic-Sine-thm1}, it is easy to see that $1+\theta<f_\theta(x)\le f_\theta(r)$ on $(0,r]$ for $\theta\le2$. The inequality~\eqref{Arc-Hyperbolic-Sine-ineq1} is thus proved.
\par
For $\theta>2$, the minimum point $x_0\in(0,\infty)$ satisfies
$$
\arcsinh x_0=\frac{x_0\bigl({\theta}/{\sqrt{x_0^2+1}\,}+1\bigr)} {\theta+1/{\sqrt{x_0^2+1}\,}}.
$$
Therefore, the minimum of the function $f_\theta(x)$ on $(0,\infty)$ equals
\begin{gather*}
\frac{\bigl(\theta+\sqrt{1+x_0^2}\,\bigr)\arcsinh x_0}{x_0} =\frac{x_0\bigl({\theta}/{\sqrt{x_0^2+1}\,}+1\bigr)} {\theta+1/{\sqrt{x_0^2+1}\,}} \cdot \frac{\bigl(\theta+\sqrt{1+x_0^2}\,\bigr)}{x_0}\\
=\frac{\bigl({\theta}/{\sqrt{x_0^2+1}\,}+1\bigr)\bigl(\theta+\sqrt{1+x_0^2}\,\bigr)} {\theta+1/{\sqrt{x_0^2+1}\,}} =\frac{\bigl(\theta+\sqrt{1+x_0^2}\,\bigr)^2} {\theta{\sqrt{x_0^2+1}\,}+1} \ge4\biggl(1-\frac1{\theta^2}\biggr).
\end{gather*}
From this, it is obtained that
$$
4\biggl(1-\frac1{\theta^2}\biggr)\le f_\theta(x)\le \max\Bigl\{\lim_{x\to0^+}f_\theta(x),f_\theta(r)\Bigr\}
$$
for $x\in(0,r]$, which implies the inequality~\eqref{Arc-Hyperbolic-Sine-ineq2}. The proof of Theorem~\ref{Arc-Hyperbolic-Sine-thm2} is thus completed.
\end{proof}

\end{document}